\documentclass{amsart}
\usepackage{ amsmath, amsthm, amsfonts, hyperref, graphicx, color, ifpdf}

\newcommand{\vx}{{\bf x}}

\newcommand{\ol}{\overline}

\newcommand{\diam}{\operatorname{diam}}

\newcommand{\be}{\begin{equation}}
\newcommand{\ee}{\end{equation}}

\newcommand{\lll}{\left}
\newcommand{\rrr}{\right}

\newcommand{\thh}{\tilde{h}}
\newcommand{\tg}{\tilde{g}}

\newcommand{\tnabla}{\tilde{\nabla}}

\newcommand{\tD}{\tilde{D}}

\newtheorem{theorem}{Theorem}[section]
\newtheorem{lemma}[theorem]{Lemma}

\theoremstyle{definition}

\theoremstyle{remark}
\newtheorem{remark}[theorem]{Remark}

\numberwithin{equation}{section}

\begin{document}
\setlength{\baselineskip}{1.2\baselineskip}

\title[The Weyl Problem In Hyperbolic Space]
{The Weyl Problem With Nonnegative Gauss Curvature In Hyperbolic Space}

\author{Jui-En Chang}

\author{Ling Xiao}




\begin{abstract}
In this paper, we discuss the isometric embedding problem in hyperbolic space with nonnegative extrinsic curvature.
We prove a priori bounds for the trace of the second fundamental form $H$ and extend the result to $n$-dimensions.
We also obtain an estimate for the gradient of the smaller principal curvature in 2 dimensions.
\end{abstract}

\maketitle

\section{Introduction} \label{Int}

In 1916, H. Weyl posed the following problem: Consider a two-sphere $S^2$ and suppose $g$
is a Riemannian metric on $S^2$ whose Gauss curvature is everywhere positive. Does there exist
a global $C^2$ isometric embedding $X:(S^2, g)\rightarrow(\mathbb{R}^3, \sigma),$ where $\sigma$ is the standard
flat metric in $\mathbb{R}^3$?

The first attempt to solve the problem was made by Weyl himself. He used the continuity method to obtain
a priori estimates up to the second derivatives. In 1953, L. Nirenberg \cite{Ni53} gave a complete solution under
the very mild hypothesis that the metric $g$ has continuous fourth derivative.

In 1964, A.V.Pogorelov \cite{Po64} obtained several important refinements to the differential geometry of 2-dimensional submanifolds $F$ of smooth 3-dimensional Riemannian manifolds $R$. The proof is based on derivative estimates which refine and complete his earlier work with A. D. Aleksandrov \cite{AP50}

The Weyl estimate was later generalized to the case of nonnegative curvature in Euclidean space by J.A. Iaia \cite{Ia92} and P. Guan, Y. Li in
\cite{GL94}. They obtained a $C^{1,1}$ embedding result for metrics of nonnegative Guass curvature; see also \cite{HZ95}
for a different approach to the $C^{1,1}$ embedding result. Later in 1999, Y. Li and G. Weinstein \cite {LW99} extended the estimates obtained in
\cite{GL94} to n-dimensions.

In this paper, we discuss the isometric embedding problem in hyperbolic space with nonnegative extrinsic Gauss curvature.
As in the Euclidian case, the image of such an embedding, if it exists, bounds a convex body.
However, the loss of strict positivity of $K$ leads to degenerate Monge-Ampere equations which arises difficulties.

By combining A.V. Pogorelov's method with the results obtained by B. Guan, J. Spruck and M.  Szapiel \cite{GSS09} and
\cite{GS11}, we can also obtain a priori bounds for principal curvatures of any strictly convex closed hypersurface with positive sectional curvature in $\mathbb{H}^{n+1}$.
Moreover, we proved that when the sectional curvature is equal to $-1$ at finitely many points, the above statement is still true.

For the case of strictly positive Gauss curvature, it's well known that the regularity exists, because  the corresponding
PDE is uniformly elliptic. Thus, if the given metric is $C^{\infty},$ then the resulting embedding is $C^{\infty}.$
However, this does not hold in the case under consideration. So, another natural question to ask is the regularity of
the embedding.

The main theorems of this paper are the following:

\begin{theorem}
\label{Intt1}
Let $g\in C^4$ be a Riemannian metric on $S^2$ with Gauss curvature satisfying\\
1) $K(P_i)=-1$, $1\leq i\leq n;$\\
2) $K(Q)> -1$ for any $Q\neq P_i,$
where $\{P_i\}\in S^2$ are finite isolated points.
Then there exists a $C^{1,1}$ isometric embedding $X:(S^2, g)\rightarrow(\mathbb{H}^3, h)$
where $h=\frac{4}{(1-|x|^2)^2}(dx_1^2+dx_2^2+dx_3^2).$
\end{theorem}

By studying the regularity of the embedding, we also prove
\begin{theorem}
\label{Intt2}
If $g\in C^5,$ under the hypothesis of Theorem \ref{Intt1} and also assume that
$\lim\sup_{Q\rightarrow P_i}|\nabla K|^2/K=C<\infty$ and $\lim\inf_{Q\rightarrow P_i} H(Q)=c>0,$
then $\kappa_1\in C^{0, 1}$ in $B_r(P_i),$ where $\kappa_1$ is the smaller of the two principal curvatures and $r>0$ sufficiently small.
\end{theorem}

We extend the results of Theorem \ref{Intt1} to higher dimensions by bounding the mean curvature $H$ in terms of the scalar curvature and its Laplacian.
This is a direct generalization of the Weyl estimate.

\begin{theorem}
\label{Intt3}
Let $g$ be a $C^4$ metric with sectional curvature $\geq -1$ and let $X: (S^n, g)\rightarrow \mathbb{H}^{n+1}$ be a $C^4$ isometric embedding.
Suppose the sectional curvature $S$ of $g$ satisfies\\
(1) $S(P_i, \chi)=-1$ for some $\chi\in\bigwedge^2T_{P_i}M;$ $1\leq i\leq n,$\\
(2) $S(Q)>-1,$ for any $Q\neq P_i,$ where $\{P_i\}\in S^n$ are finite isolated points.\\
Let $H$ be the trace of the second fundamental form of $X,$ and
let $R$ be the extrinsic scalar curvature of $g.$ Then, the following inequality holds:
\be\label{int1}
H^2\leq C_1|\Delta R|+C_2(R^2+R).
\ee
where $C_1,$ $C_2$ depends only on the metric $g$ and dimension $n.$
\end{theorem}
An outline of the contents of the paper are as follows. Section \ref{mce} contains the estimate of mean curvature and proves
Theorem \ref{Intt1}. Section \ref{dmh} introduces the relations between different models of hyperbolic space, and also state
how to transform the coordinates of a small neighborhood to simplify the calculations, which will be used in Section \ref{pot}
to prove Theorem \ref{Intt2} and also in Section \ref{apb} to prove Theorem \ref{Intt3}. The use of the transform map between
two different models is unusual, but seems to be necessary in proving the partial third derivative estimates.
\bigskip

\section{A mean curvature estimate} \label{mce}
\begin{theorem}
\label{mcet1}
Let $g\in C^4$ be a Riemannian metric on $S^2$ with Gauss curvature satisfying\\
1) $K(P_i)=-1,$ $1\leq i\leq n,$\\
2) $K(Q)> -1$ for any $Q\neq P_i,$
where $\{P_i\}\in S^2$ are finite isolated points.
Then, there exists a $C^{1,1}$ isometric embedding $X:(S^2, g)\rightarrow(\mathbb{H}^3, h)$
where $h=\frac{4}{(1-|x|^2)^2}(dx_1^2+dx_2^2+dx_3^2).$
\end{theorem}
\begin{proof}
Following \cite{GL94} we first approximate $g^0$ in $C^4$ by a sequence of $C^\infty$ metric $g^{\epsilon}$
with corresponding intrinsic Guass curvature $\{K^\epsilon\},$ such that $K^\epsilon>-1$ everywhere.
Then, we can apply the result from \cite{Po64} to $g^\epsilon$ and therefore obtain a sequence of $C^\infty$ isometric embeddings
\[X^\epsilon:(S^2, g^0)\rightarrow(\mathbb{H}^3, h).\]
It's not difficult to see that there exists constants $\alpha,$ $\beta>0$ (independent of $\epsilon$), such that
 for all $\epsilon>0,$
 \[-1<K^{\epsilon}<\alpha\]
 and
 \[diam\lll(X^{\epsilon}\rrr)<\beta.\]
 We immediately have that
 \[\|X^{\epsilon}\|_{C^0}\leq C.\]

 In local coordinates,
 \[g^0=E^0du^2+2F^0dudv+G^0dv^2,\]
 \[g^\epsilon=E^\epsilon du^2+2F^\epsilon dudv+G^\epsilon dv^2.\]
 We already know that $X^\epsilon: (S^2, g^\epsilon)\rightarrow (\mathbb{H}^3, h)$
 is an isometric embedding, so we have
 \[\lll<X^\epsilon_u,X^\epsilon_u\rrr>_{\mathbb{H}^3}=E^\epsilon,\; \lll<X^\epsilon_u,X^\epsilon_v\rrr>_{\mathbb{H}^3}=F^\epsilon,\;
\lll<X^\epsilon_v,X^\epsilon_v\rrr>_{\mathbb{H}^3}=G^\epsilon.\]
It follows easily that
\[\|\nabla_{g^0}X^\epsilon\|_{C^0}\leq C,\]
where C is independent of $\epsilon.$

 The following will be devoted to establishing a bound on
$\|\nabla^2_{g^0}X^\epsilon\|_{C^0},$ which is independent of $\epsilon.$ Once we obtain such a bound, the limit of $X^\epsilon$ as $\epsilon\rightarrow 0$
will be a $C^{1,1}$ isometric embedding of $g^0.$ For convenience, in the following we drop the dependence on $\epsilon$
in our notation.

Next, we will prove
\begin{equation}
\label{eqm1}
\max_{S^2}H\leq C,
\end{equation}
where $C$ is a constant independent of $\epsilon.$

In hyperbolic space we have
\be
\label{eqm2}
\frac{4}{(1-|x|^2)^2}\left<X_u, X_u\right>_{\mathbb{R}^3}=\left<X_u, X_u\right>_{\mathbb{H}^3}=E=g_{11},
\ee
\be
\label{eqm3}
\frac{4}{(1-|x|^2)^2}\left<X_u, X_v\right>_{\mathbb{R}^3}=\left<X_u, X_v\right>_{\mathbb{H}^3}=F=g_{12},
\ee
\be
\label{eqm4}
\frac{4}{(1-|x|^2)^2}\left<X_v, X_v\right>_{\mathbb{R}^3}=\left<X_v, X_v\right>_{\mathbb{H}^3}=G=g_{22}.
\ee
Let the orientation be chosen so that the inner unit normal is given by
\be
\label{eqm5}
\ol{X}=\frac{X_u\times X_v}{\sqrt{EG-F^2}}=\frac{2}{1-|x|^2}\ol{X}_E,
\ee
where $\ol{X}_E$ is the Euclidean unit normal.
The second fundamental form is then given by
\[II=Ldu^2+2Mdudv+Ndv^2\]
where
\be\label{eqm6}
L=-\left<X_u, \ol{X}_u\right>_{\mathbb{H}^3}=\left<\nabla_{X_u}X_u, \ol{X}\right>_{\mathbb{H}^3},
\ee
\be\label{eqm7}
M=-\left<X_v, \ol{X}_u\right>_{\mathbb{H}^3}=\left<\nabla_{X_u}X_v, \ol{X}\right>_{\mathbb{H}^3},
\ee
\be\label{eqm8}
N=-\left<X_v, \ol{X}_v\right>_{\mathbb{H}^3}=\left<\nabla_{X_v}X_v, \ol{X}\right>_{\mathbb{H}^3}.
\ee
Hence the intrinsic Gauss and mean curvature are:
\be\label{eqm9}
K=-1+\frac{LN-M^2}{EG-F^2}
\ee
and
\be\label{eqm10}
H=\frac{1}{2}\frac{GL-2FM+EN}{EG-F^2}.
\ee
The Gauss equation takes the form:
\be\label{eqm11}
\nabla_{X_u}X_u=\Gamma^1_{11}X_u+\Gamma^2_{11}X_v+L\ol{X},
\ee
\be\label{eqm12}
\nabla_{X_u}X_v=\Gamma^1_{12}X_u+\Gamma^2_{12}X_v+M\ol{X},
\ee
\be\label{eqm13}
\nabla_{X_v}X_v=\Gamma^1_{22}X_u+\Gamma^2_{22}X_v+N\ol{X},
\ee
where $\Gamma^k_{ij}=\frac{1}{2}g^{kl}\left(\partial_ig_{jl}+\partial_jg_{il}-\partial_lg_{ij}\right)$
with $\partial_1=\partial_u$ and $\partial_2=\partial_v.$
The Weingarten equations take the form:
\be\label{eqm14}
-\ol{X}_u=L^1_1X_u+L^2_1X_v,
\ee
\be\label{eqm15}
-\ol{X}_v=L^1_2X_u+L^2_2X_v,
\ee
where $\{L^i_j\}$ are expressions involving $L, M, N$ and $E, F, G.$
The Mainardi-Codazzi equations take the form:
\be\label{eqm16}
L_v-M_u=L\Gamma^1_{12}+M\left(\Gamma^2_{12}-\Gamma^1_{11}\right)-N\Gamma^2_{11},
\ee
\be\label{eqm17}
M_v-N_u=L\Gamma^1_{22}+M\left(\Gamma^2_{22}-\Gamma^1_{12}\right)-N\Gamma^2_{12}.
\ee

Let $\rho=\frac{2}{1-|x|^2}.$ We consider the following function on $S^2,$
\[f=e^{\alpha\rho}H,\]
with $\alpha>0$ to be determined later. Without loss of generality, we assume there is only one singular point
$P_0\in S^2$ such that $\lim_{\epsilon\rightarrow 0}K^\epsilon(P_0)=-1.$
Let $\delta=\inf_{S^2/B_r(P_0)}K>-1$ and be independent of $\epsilon.$ Then, by a Theorem in \cite{Po64} (page 216),
we have
\be\label{eqm18}H\leq C_0,\,\mbox{for $\forall P\in S^2\backslash B_r(P_0)$},
\ee
where $C_0$ depends only upon the metric of $S^2\backslash B_r(P_0)$
and the metric of the space.

Next, we will focus on estimating the curvature inside $B_r(P_0).$
Since we use the ball model for hyperbolic space, we can always choose our origin very close to
$X(P_0)$ such that for any point $P\in X(B_r(P_0)\cap S^2)$
we have $|P|_{\mathbb{R}^3}^2=p_1^2+p_2^2+p_3^2\leq\frac{1}{100}.$
(We always assume $r$ is small.)
Restricting the function $f$ on $X(B_r(P_0)\cap S^2),$ we can see that
if $M=\max_{P\in X(B_r(P_0)\cap S^2)}f$ is achieved on the boundary, then we would have
\[e^{\alpha\rho}H\leq e^{3\alpha}C_0\;\mbox{for $\forall P\in X(B_r(P_0)\cap S^2)$}.\]

Therefore, we assume M is achieved at an interior point $Q.$ Let's write the metric
$g=g^\epsilon$ near $Q$ in conformal coordinates:
\[g=e^{2h}(du^2+dv^2),\]
where $(u, v)=(0, 0)$ corresponds to $Q,$ and
\[h=\partial_uh=\partial_vh=0\,\mbox{at $(0, 0)$}.\]
The intrinsic Gauss and mean curvatures become
\be\label{eqm19}
K=-1+\frac{LN-M^2}{e^{4h}},
\ee
\be\label{eqm20}
H=\frac{L+N}{2e^{2h}},
\ee
\be\label{eqm21}
K=-\frac{\tilde{\triangle}h}{e^{2h}},
\ee
where $\tilde{\triangle}=\partial_{11}+\partial_{22},$
also, the Mainardi-Codazzi equations
\be\label{eqm22}
L_2-M_1=h_2(L+N)=2He^{2h}h_2,
\ee
\be\label{eqm23}
M_2-N_1=-h_1(L+N)=-2He^{2h}h_1.
\ee
Clearly $\triangle_gK=e^{-2h}\tilde{\triangle}K.$
Differentiating \eqref{eqm19}, we have
\be\label{eqm24}
K_1=\frac{1}{e^{4h}}\left(L_1N+LN_1-2MM_1\right)-4h_1\left(K+1\right),
\ee
\be\label{eqm25}
\begin{aligned}
K_{11}&=\frac{1}{e^{4h}}(L_{11}N+LN_{11}+2N_1L_1-2M_1^2-2MM_{11})\\
&-8h_1K_1-4h_{11}(K+1)-16h_1^2(K+1),
\end{aligned}
\ee
\be\label{eqm26}
K_2=\frac{1}{e^{4h}}(L_2N+LN_2-2MM_2)-4h_2(K+1)
\ee
\be\label{eqm27}
\begin{aligned}
K_{22}&=\frac{1}{e^{4h}}(L_{22}N+LN_{22}+2N_2L_2-2M_2^2-2MM_{22})\\
&-8h_2K_2-4h_{22}(K+1)-16h_2^2(K+1).
\end{aligned}
\ee
Apply $\partial_2$ to \eqref{eqm22}, $\partial_1$ to \eqref{eqm23}, and add together
\be\label{eqm28}
(-L_2+2He^{2h}h_2)_2=(-N_1+2He^{2h}h_1)_1
\ee
which yields
\be\label{eqm29}
N_{11}=L_{22}+2H_1h_1e^{2h}-2H_2h_2e^{2h}+2(h_{11}-h_{22})He^{2h}+4(h_1^2-h_2^2)He^{2h}.
\ee
Differentiating \eqref{eqm20} gives
\be\label{eqm30}
2H_1=\frac{L_1+N_1}{e^{2h}}-4h_1H,
\ee
\be\label{eqm31}
2H_{11}=\frac{L_{11}+N_{11}}{e^{2h}}-8h_1H_1-8h_1^2H-4h_{11}H,
\ee
\be\label{eqm32}
2H_2=\frac{L_2+N_2}{e^{2h}}-4h_2H,
\ee
\be\label{eqm33}
2H_{12}=\frac{L_{12}+N_{12}}{e^{2h}}-4h_2H_1-4h_{12}H-4h_1H_2-8h_1h_2H,
\ee
\be\label{eqm34}
2H_{22}=\frac{L_{22}+N_{22}}{e^{2h}}-8h_2H_2-8h_2^2H-4h_{22}H.
\ee
Applying $\partial_1$ to \eqref{eqm22} and $\partial_2$ to \eqref{eqm23}, we obtain
\be\label{eqm35}
M_{11}=L_{12}-2H_1e^{2h}h_2-2He^{2h}h_{12}-4Hh_1h_2e^{2h},
\ee
\be\label{eqm36}
M_{22}=N_{12}-2H_2e^{2h}h_1-2He^{2h}h_{12}-4Hh_1h_2e^{2h}.
\ee
So, we have
\be\label{eqm37}
\begin{aligned}
e^{6h}\triangle_gK&=e^{4h}\tilde{\triangle}K\\
&=N(L_{11}+L_{22})
+L(N_{11}+N_{22})-2M(M_{11}+M_{22})\\
&+2(L_1N_1+L_2N_2-M_1^2-M_2^2)\\
&-e^{4h}[8h_1K_1+8h_2K_2+4(h_{11}+h_{22})(K+1)+16(h_1^2+h_2^2)(K+1)].\\
\end{aligned}
\ee
Substitute \eqref{eqm29}, \eqref{eqm35}, and \eqref{eqm36} into \eqref{eqm37}
we get
\be\label{eqm38}
\begin{aligned}
e^{6h}\triangle_gK&=N[L_{11}+N_{11}-2H_1h_1e^{2h}+2H_2h_2e^{2h}-2(h_{11}-h_{22})He^{2h}\\
&-4(h_1^2-h_2^2)He^{2h}]\\
&+L[L_{22}+N_{22}+2H_1h_1e^{2h}-2H_2h_2e^{2h}+2(h_{11}-h_{22})He^{2h}\\
&+4(h_1^2-h_2^2)He^{2h}]\\
&-2M[L_{12}+N_{12}-2(H_1h_2+H_2h_1)e^{2h}-4He^{2h}h_{12}-8Hh_1h_2e^{2h}]\\
&+2(L_1N_1+L_2N_2-M_1^2-M_2^2)\\
&-e^{4h}[8h_1K_1+8h_2K_2+4(h_{11}+h_{22})(K+1)+16(h_1^2+h_2^2)(K+1)].\\
\end{aligned}
\ee
Plug in \eqref{eqm31}, \eqref{eqm33}, and \eqref{eqm34} we obtain
\be\label{eqm39}
\begin{aligned}
e^{6h}\triangle_gK&=N[2e^{2h}H_{11}+6H_1h_1e^{2h}+2H_2h_2e^{2h}+2(h_{11}+h_{22})He^{2h}\\
&+4(h_1^2+h_2^2)He^{2h}]\\
&+L[2e^{2h}H_{22}+6h_2H_2e^{2h}+2H_1h_1e^{2h}+2(h_{11}+h_{22})He^{2h}\\
&+4(h_1^2+h_2^2)He^{2h}]\\
&-2M[2e^{2h}H_{12}+2(H_1h_2+H_2h_1)e^{2h}]\\
&+2(L_1N_1+L_2N_2-M_1^2-M_2^2)\\
&-e^{4h}[8h_1K_1+8h_2K_2+4(h_{11}+h_{22})(K+1)+16(h_1^2+h_2^2)(K+1)].\\
\end{aligned}
\ee
Regrouping the terms and using \eqref{eqm21}, we get
\be\label{eqm40}
\begin{aligned}
e^{6h}\triangle_gK&=2e^{2h}(NH_{11}-2MH_{12}+LH_{22})
+2(L_1N_1+L_2N_2-M_1^2-M_2^2)\\
&+N[6H_1h_1e^{2h}+2H_2h_2e^{2h}+2(h_{11}+h_{22})He^{2h}+4(h_1^2+h_2^2)He^{2h}]\\
&+L[6H_2h_2e^{2h}+2H_1h_1e^{2h}+2(h_{11}+h_{22})He^{2h}+4(h_1^2+h_2^2)He^{2h}]\\
&-4M(H_1h_2+H_2h_1)e^{2h}\\
&-e^{4h}[8h_1K_1+8h_2K_2-4K(K+1)e^{-2h}+16(h_1^2+h_2^2)(K+1)].\\
\end{aligned}
\ee
From \eqref{eqm30} and \eqref{eqm32}, we derive that
\be\label{eqm41}
L_1N_1\leq\frac{1}{4}(L_1+N_1)^2=e^{4h}(H_1+2h_1H)^2,
\ee
\be\label{eqm42}
L_2N_2\leq\frac{1}{4}(L_2+N_2)^2=e^{4h}(H_2+2h_2H)^2.
\ee
By \eqref{eqm20}, \eqref{eqm21}, \eqref{eqm41}, \eqref{eqm42} and that at point $Q,$ $h=h_1=h_2=0$ we have
\be\label{eqm43}
\triangle_gK\leq 2(NH_{11}-2MH_{12}+LH_{22})+2(H_1^2+H_2^2)-4KH^2+4K(K+1).
\ee
By assumption, we have at point $Q$,
\be\label{eqm44}
f_i(Q)=0,\,\mbox{i=1,2}.
\ee
Thus
\be\label{eqm45}
f_i=\alpha e^{\alpha\rho}H\rho_i+e^{\alpha\rho}H_i=0
\ee
and
\be\label{eqm46}
f_{ij}=e^{\alpha\rho}(\alpha^2H\rho_i\rho_j+\alpha H_i\rho_j+\alpha H_j\rho_i+\alpha H\rho_{ij}+H_{ij}),
\ee
where $i,j=1,2.$
Therefore, the following hold at $Q$,
\be\label{eqm47}
H_1=-\alpha\rho_1H,\,\,\,H_2=-\alpha\rho_2H,
\ee
\be\label{eqm48}
H_{11}=-\alpha H\rho_{11}+\alpha^2\rho_1^2H+f_{11}e^{-\alpha\rho},
\ee
\be\label{eqm49}
H_{12}=-\alpha H\rho_{12}+\alpha^2\rho_1\rho_2H+f_{12}e^{-\alpha\rho},
\ee
\be\label{eqm50}
H_{22}=-\alpha H\rho_{22}+\alpha^2\rho_2^2H+f_{22}e^{-\alpha\rho}.
\ee

Since at $Q$, $f$ achieves a local maximum, we also have
\[\left\{f_{ij}(Q)\right\}\leq0,\,\,1\leq i, j\leq 2.\]
Thus
\be\label{eqm51}
\begin{aligned}
NH_{11}-2MH_{12}+LH_{22}&\leq H[L(-\alpha\rho_{22}+\alpha^2\rho_2^2)-2M(-\alpha\rho_{12}+\alpha^2\rho_1\rho_2)\\
&+N(-\alpha\rho_{11}+\alpha^2\rho_1^2)]
\end{aligned}
\ee
and
\be\label{eqm52}
\left(H_1^2+H_2^2\right)=H^2\alpha^2(\rho_1^2+\rho_2^2).
\ee
Combining \eqref{eqm43}, \eqref{eqm51}, and \eqref{eqm52} we obtain
\be\label{eqm53}
\begin{aligned}
\triangle_gK&\leq-2\alpha H(L\rho_{22}-2M\rho_{12}+N\rho_{11})+2\alpha^2 H(L+N)(\rho_1^2+\rho_2^2)\\
&+2H^2\alpha^2(\rho_1^2+\rho_2^2)+4H^2+4K(K+1)\\
&=-2\alpha H(L\rho_{22}-2M\rho_{12}+N\rho_{11})+6H^2\alpha^2(\rho_1^2+\rho_2^2)+4H^2+4K(K+1).\\
\end{aligned}
\ee

Now, let $\vec{r}=x_1\frac{\partial}{\partial x_1}+x_2\frac{\partial}{\partial x_2}+x_3\frac{\partial}{\partial x_3}.$
Then, we have
\be\label{eqm54}
\rho_u=\lll<\vec{r}, X_u\rrr>_{\mathbb{H}^3},\;\rho_v=\lll<\vec{r}, X_v\rrr>_{\mathbb{H}^3},
\ee
\be\label{eqm55}
\rho_u^2+\rho_v^2\leq \lll<\vec{r}, \vec{r}\rrr>_{\mathbb{H}^3},
\ee
\be\label{eqm56}
\rho_{uu}=\frac{1+|x|^2}{1-|x|^2}E+\bar{\Gamma}^1_{11}\lll<\frac{\partial}{\partial u}, \vec{r}\rrr>_{\mathbb{H}^3}
+\bar{\Gamma}^2_{11}\lll<\frac{\partial}{\partial v},\vec{r}\rrr>_{\mathbb{H}^3}+L\lll<\ol{X}, \vec{r}\rrr>_{\mathbb{H}^3},
\ee
\be\label{eqm57}
\rho_{uv}=\frac{1+|x|^2}{1-|x|^2}F+\bar{\Gamma}^1_{12}\lll<\frac{\partial}{\partial u}, \vec{r}\rrr>_{\mathbb{H}^3}
+\bar{\Gamma}^2_{12}\lll<\frac{\partial}{\partial v},\vec{r}\rrr>_{\mathbb{H}^3}+M\lll<\ol{X}, \vec{r}\rrr>_{\mathbb{H}^3},
\ee
and
\be\label{eqm58}
\rho_{vv}=\frac{1+|x|^2}{1-|x|^2}G+\bar{\Gamma}^1_{22}\lll<\frac{\partial}{\partial u}, \vec{r}\rrr>_{\mathbb{H}^3}
+\bar{\Gamma}^2_{22}\lll<\frac{\partial}{\partial v},\vec{r}\rrr>_{\mathbb{H}^3}+N\lll<\ol{X}, \vec{r}\rrr>_{\mathbb{H}^3}.
\ee
Therefore, at point $Q,$ we have
\be\label{eqm59}
\begin{aligned}
\Delta_gK&\leq-2\alpha H\lll(L\rho_{22}+N\rho_{11}\rrr)+6H^2\alpha^2\lll(\rho_1^2+\rho_2^2\rrr)+4H^2+4K(K+1)\\
&\leq-2\alpha H\lll[L\lll(\frac{1+|x|^2}{1-|x|^2}G+N\lll<\ol{X},\vec{r}\rrr>_{\mathbb{H}^3}\rrr)
+N\lll(\frac{1+|x|^2}{1-|x|^2}E+L\lll<\ol{X},\vec{r}\rrr>_{\mathbb{H}^3}\rrr)\rrr]\\
&+6\alpha^2 H^2\lll<\vec{r},\vec{r}\rrr>_{H^3}+4H^2+4K(K+1)\\
&\leq-4\alpha H^2\frac{1+|x|^2}{1-|x|^2}+6\alpha^2 H^2\frac{4|x|^2}{\lll(1-|x|^2\rrr)^2}+4H^2+4K(K+1)\\
&=4\lll(1-\alpha\frac{1+|x|^2}{1-|x|^2}+6\frac{\alpha^2|x|^2}{\lll(1-|x|^2\rrr)^2}\rrr)H^2+C(K).\\
\end{aligned}
\ee
Now choose $\alpha=2,$ by our assumption, at $Q,$ $|x|^2<\frac{1}{100}.$ Hence, we get
\be\label{eqm60}
\Delta_gK\leq -H^2+C(K),
\ee
which implies
\be\label{eqm61}
H^2\leq C(K)-\Delta_gK.
\ee
Therefore,
\be\label{eqm62}
\max_{\overline{X(B_r(P_0)\cap S^2)}}f\leq e^{2\rho(Q)}H(Q)\leq C,
\ee
from which we conclude,
\[\max_{S^2}H\leq C\]
where C is independent of $\epsilon.$
This completes the proof.
\end{proof}

\begin{remark}
From equation \eqref{eqm59}, we can see that if $|x|^2<\frac{\alpha-1}{6\alpha^2+\alpha-1}$
for some $\alpha>1$ (i.e. $\diam (X^\epsilon)$ is small in the hyperbolic space),
then the existence of a $C^{1,1}$ embedding is true as long as $K\geq -1$ on $S^2.$
\end{remark}

\bigskip

\section{Different models of hyperbolic space}\label{dmh}
\subsection{Ball model and upper half-space model}\label{dmhb}

Some times it's easier to do calculations in the upper half-space model instead of the ball model.
Therefore, we can use the transformation function between the coordinates:\\

\includegraphics[scale=0.5]{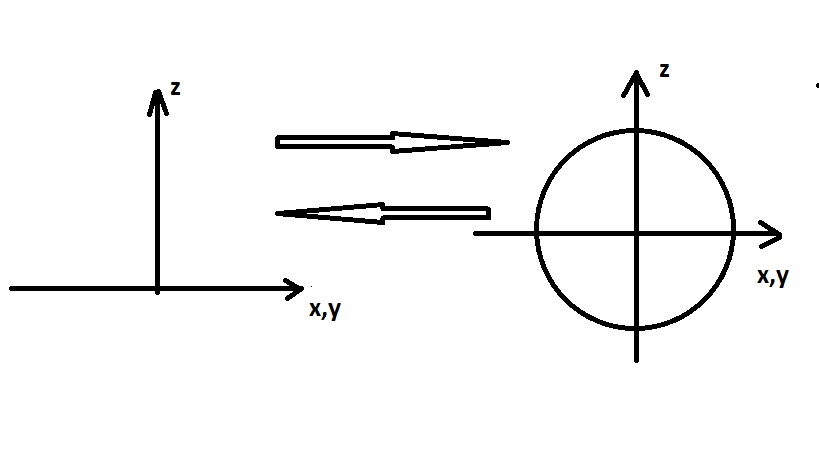}

\[\phi(x,y,z)=\frac{1}{x^2+y^2+(z+1)^2}(2x,2y,x^2+y^2+z^2-1)\]
\[\phi^{-1}(x,y,z)=\frac{1}{x^2+y^2+(z-1)^2}(2x,2y,1-(x^2+y^2+z^2))\]

In the next section, we are going to get estimates around singular points. So it will be helpful if we can
transform coordinates of a neighborhood of the singular point such that it can be represented as a
graph in the upper half-space model.

For any point $P$ belonging to our submanifold, we want to make a neighborhood of $P$
to be graphical over the $xy$-plane in the upper half-space model. We can achieve this by the following procedure:

Step1. Use the inward normal vector $N$ of $P$, we can get a geodesic $\gamma(t)$ starting from $P$
with tangent $N$;

Step2. Rotate $P$ such that $N$ is parallel to the $\frac{\partial}{\partial z}$ direction.
Then choose some $d>0$, such that $\gamma(d)$ is the origin, and $\gamma([0,d])$ lies on the
z-axis in the unit ball coordinates.

Step3. Use the transformation $\phi^{-1}$ above to get $P\in z-$axis in the upper half-space coordinates,
and the normal $N$ is parallel to $\frac{\partial}{\partial z}.$ Therefore, there is a neighborhood of $P$ which
can be written as the graph of a function $z=u(x, y)$ over the $xy$-plane. Moreover, $\nabla z=0$ at $(0, 0).$

Note that the above procedure is true for n-dimensions.

\subsection{Vertical graph in upper half-space model}\label{dmhf}
We will use the half-space model
\[\mathbb{H}^{n+1}=\{(x, x_{n+1})\in \mathbb{R}^{n+1}:x_{n+1}>0\}\]
equipped with the hyperbolic metric
\be\label{dmhf1}
ds^2=\frac{\sum_{i+1}^{n+1}dx_i^2}{x_{n+1}^2}.
\ee

If $\Sigma$ is the graph of a function $u(x)$
and $x\in\Omega\subset R^n\times\{0\},$
\[\Sigma=\{(x, x_{n+1}):x\in\Omega, x_{n+1}=u(x)\},\]
then the coordinate vector fields and downward unit normal are given by
\[X_i=e_i+u_ie_{n+1},\;\mathbf{n} =u\nu=u\frac{u_ie_i-e_{n+1}}{w},\]
where $w=\sqrt{1+|\nabla u|^2}$ and $\nu$ is the Euclidean downward unit normal to $\Sigma.$
The first fundamental form $g_{ij}$ is then given by
\be\label{dmhf2}
g_{ij}=\lll<X_i, X_j\rrr>_{\mathbb{H}^{n+1}}=\frac{1}{u^2}(\delta_{ij}+u_iu_j)=\frac{g^e_{ij}}{u^2}.
\ee
To compute the second fundamental form $h_{ij},$ we use
\be\label{dmhf3}
\Gamma^k_{ij}=\frac{1}{x_{n+1}}(-\delta_{jk}\delta_{in+1}-\delta_{ik}\delta_{jn+1}+\delta_{ij}\delta_{kn+1})
\ee
to obtain
\be\label{dmhf4}
\nabla_{X_i}X_j=\lll(\frac{\delta_{ij}}{x_{n+1}}+u_{ij}-\frac{u_iu_j}{x_{n+1}}\rrr)e_{n+1}
-\frac{u_je_i+u_ie_j}{x_{n+1}}.
\ee
Then
\be\label{dmhf5}
h_{ij}=\lll<\nabla_{X_i}X_j, u\nu\rrr>_{\mathbb{H}^{n+1}}=\frac{-1}{u^2w}(\delta_{ij}+u_iu_j+uu_{ij}).
\ee
\bigskip

\section{Proof of Theorem 1.2}\label{pot}
\subsection{Partial third derivative estimates}\label{ptd}
In this subsection, we will establish the following lemma. The argument is
based on the argument originally given by E. Calabi \cite{Ca58}, see also \cite{CNS84} and \cite{Ia92}.
Note that in this section, for convenience, we denote the extrinsic Gauss curvature by $K$.
\label{ptd}
\begin{lemma}
\label{ptdl1}
In a neighborhood of $(0,0),$
\be\label{ptd1}
L\sigma\geq\frac{\sigma^2}{2}-\frac{C}{K^4}
\ee
holds, where
\be\label{ptd2}
L\sigma=\rho^{ij}\sigma_{ij}+\frac{4u_i\sigma_i}{\lll(1+|\nabla u|^2\rrr)u},
\ee
\[\sigma=\rho^{kl}\rho^{pq}\rho^{rs}\rho_{kpr}\rho_{lqs},\]
\[\rho=-\frac{x^2+u^2}{2},\]
and $C$ is a constant depending on the maximum of $\frac{1}{w}$ near $(0,0),$
the maximum of $|K|_{C^2},$ the maximum of $K^{1/2}|\nabla^3 K|,$ the maximum of $|\nabla K|^2/K,$
and the maximum of mean curvature $H.$
\end{lemma}
\begin{remark}\label{ptdr1}
The reason for the presence of the term $|\nabla K|^2/K$ instead of $1/K$
is that for a large class of functions satisfying
$K(0,0)=0$ and $K(Q)\neq 0$ when $Q\neq (0,0),$ the quantity $|\nabla K|^2/K$ is bounded while $1/K$ is not.
\end{remark}
\begin{proof}
In the following let $\rho=-\frac{x^2+u^2}{2}$. By \eqref{dmhf5}
we have
\be\label{ptd3}
u^2wh_{ij}=\rho_{ij},\;\mbox{where $w=\sqrt{1+|\nabla u|^2}.$}
\ee
Hence,
\be\label{ptd4}
\frac{\det\rho_{ij}}{\det g}=\frac{u^4w^2\det{h_{ij}}}{\det g}=u^4w^2K.
\ee
Since $g_{ij}=\frac{1}{u^2}\lll(\delta_{ij}+u_iu_j\rrr),$
\[\det g=\frac{1}{u^4}\left| \begin{array}{cc} 1+u_1^2 & u_1u_2 \\ u_2u_1 & 1+u_2^2 \end{array} \right|=\frac{w^2}{u^4},\]
\be\label{ptd5}
\det\rho_{ij}=w^4K=K\lll(1+|\nabla u|^2\rrr)^2.
\ee
Denote $m=K\lll(1+|\nabla u|^2\rrr)^2,$ then we have
\[\log\lll(\det\rho_{ij}\rrr)=\log m.\]
Differentiating three times gives
\be\label{ptd6}
\rho^{ij}\rho_{ijk}=\lll(\log m\rrr)_k,
\ee
\be\label{ptd7}
\rho^{ij}\rho_{ijkp}=\rho^{ia}\rho_{abp}\rho^{bj}\rho_{ijk}+\lll(\log m\rrr)_{kp},
\ee
\be\label{ptd8}
\begin{aligned}
\rho^{ij}\rho_{ijkpr}&=\rho^{ia}\rho_{abr}\rho^{bj}\rho_{ijkp}+\rho^{ia}\rho_{abpr}\rho^{bj}\rho_{ijk}
+\rho^{ia}\rho_{abp}\rho^{bj}\rho_{ijkr}\\
&-2\rho^{ic}\rho_{cdr}\rho^{da}\rho_{abp}\rho^{bj}\rho_{ijk}+\lll(\log m\rrr)_{kpr}\\
\end{aligned}
\ee
Next, consider
\be\label{ptd9}
\sigma=\rho^{kl}\rho^{pq}\rho^{rs}\rho_{kpr}\rho_{lqs},
\ee
Differentiating it with respect to $x_i$ we get,
\be\label{ptd10}
\begin{aligned}
\sigma_i&=\lll(-\rho^{ka}\rho_{abi}\rho^{bl}\rho^{pq}\rho^{rs}-\rho^{kl}\rho^{pa}\rho_{abi}\rho^{bq}\rho^{rs}\right.
\left.-\rho^{kl}\rho^{pq}\rho^{ra}\rho_{abi}\rho^{bs}\rrr)\rho_{kpr}\rho_{lqs}\\
&+2\rho^{kl}\rho^{pq}\rho^{rs}\rho_{kpri}\rho_{lqs}.
\end{aligned}
\ee
Now fix a point $P$ and rotate about $P$ so that $\{\rho_{ij}\}$
is diagonal at $P.$ Thus, in these coordinates at the point $P,$ we have
\be\label{ptd11}
\begin{aligned}
\rho^{ij}\sigma_{ij}&=2\frac{\rho_{kpr}\rho_{kprii}}{\rho_{ii}\rho_{kk}\rho_{pp}\rho_{rr}}+2\frac{\rho_{kpri}^2}{\rho_{ii}\rho_{kk}\rho_{pp}\rho_{rr}}
-8\frac{\rho_{kli}\rho_{lpr}\rho_{kpri}}{\rho_{ii}\rho_{kk}\rho_{ll}\rho_{pp}\rho_{rr}}\\
&-4\frac{\rho_{rsi}\rho_{kps}\rho_{kpri}}{\rho_{ii}\rho_{kk}\rho_{pp}\rho_{rr}\rho_{ss}}
-2\frac{\rho_{kpr}\rho_{lpr}\rho_{klii}}{\rho_{ii}\rho_{kk}\rho_{ll}\rho_{pp}\rho_{rr}}
-\frac{\rho_{kpr}\rho_{kps}\rho_{rsii}}{\rho_{ii}\rho_{kk}\rho_{pp}\rho_{rr}\rho_{ss}}\\
&+4\frac{\rho_{ali}\rho_{kai}\rho_{kpr}\rho_{lpr}}{\rho_{aa}\rho_{ii}\rho_{kk}\rho_{ll}\rho_{pp}\rho_{rr}}
+2\frac{\rho_{rai}\rho_{asi}\rho_{kpr}\rho_{kps}}{\rho_{aa}\rho_{ii}\rho_{kk}\rho_{pp}\rho_{rr}\rho_{ss}}\\
&+2\frac{\rho_{kli}\rho_{pqi}\rho_{kpr}\rho_{lqr}}{\rho_{ii}\rho_{kk}\rho_{ll}\rho_{pp}\rho_{qq}\rho_{rr}}
+4\frac{\rho_{kli}\rho_{rsi}\rho_{kpr}\rho_{lps}}{\rho_{ii}\rho_{kk}\rho_{ll}\rho_{pp}\rho_{rr}\rho_{ss}}.
\end{aligned}
\ee
Consider the first term in \eqref{ptd11} and use \eqref{ptd8} to get
\be\label{ptd12}
\begin{aligned}
2\frac{\rho_{kpr}\rho_{kprii}}{\rho_{ii}\rho_{kk}\rho_{pp}\rho_{rr}}&=
2\frac{\rho_{kpr}}{\rho_{kk}\rho_{pp}\rho_{rr}}\frac{\rho_{iikpr}}{\rho_{ii}}\\
&=2\frac{\rho_{kpr}}{\rho_{kk}\rho_{pp}\rho_{rr}}\lll[\frac{\rho_{abr}\rho_{abkp}}{\rho_{aa}\rho_{bb}}\right.
\left.\frac{\rho_{abpr}\rho_{abk}}{\rho_{aa}\rho_{bb}}+\frac{\rho_{abp}\rho_{abkr}}{\rho_{aa}\rho_{bb}}\right.\\
&\left.-2\frac{\rho_{cdr}\rho_{dbp}\rho_{cbk}}{\rho_{cc}\rho_{dd}\rho_{bb}}+\lll(\log m\rrr)_{kpr}\right].\\
\end{aligned}
\ee
Thus,
\be\label{ptd13}
\begin{aligned}
\rho^{ij}\sigma_{ij}&=2\frac{\rho_{kpr}\rho_{abr}\rho_{abkp}}{\rho_{aa}\rho_{bb}\rho_{kk}\rho_{pp}\rho_{rr}}
+2\frac{\rho_{abk}\rho_{kpr}\rho_{abpr}}{\rho_{aa}\rho_{bb}\rho_{kk}\rho_{pp}\rho_{rr}}
+2\frac{\rho_{abp}\rho_{kpr}\rho_{abkr}}{\rho_{aa}\rho_{bb}\rho_{kk}\rho_{pp}\rho_{rr}}\\
&-4\frac{\rho_{cdr}\rho_{dbp}\rho_{cbk}\rho_{kpr}}{\rho_{bb}\rho_{cc}\rho_{dd}\rho_{kk}\rho_{pp}\rho_{rr}}
+2\frac{\rho^2_{kpri}}{\rho_{ii}\rho_{kk}\rho_{pp}\rho_{rr}}
-8\frac{\rho_{kli}\rho_{lpr}\rho_{kpri}}{\rho_{ii}\rho_{kk}\rho_{ll}\rho_{pp}\rho_{rr}}\\
&-4\frac{\rho_{rsi}\rho_{kps}\rho_{kpri}}{\rho_{ii}\rho_{kk}\rho_{pp}\rho_{rr}\rho_{ss}}
-2\frac{\rho_{kpr}\rho_{lpr}\rho_{klii}}{\rho_{ii}\rho_{kk}\rho_{ll}\rho_{pp}\rho_{rr}}
-\frac{\rho_{kpr}\rho_{kps}\rho_{rsii}}{\rho_{ii}\rho_{kk}\rho_{pp}\rho_{rr}\rho_{ss}}\\
&+4\frac{\rho_{ali}\rho_{kai}\rho_{kpr}\rho_{lpr}}{\rho_{aa}\rho_{ii}\rho_{kk}\rho_{ll}\rho_{pp}\rho_{rr}}
+2\frac{\rho_{rai}\rho_{asi}\rho_{kpr}\rho_{kps}}{\rho_{aa}\rho_{ii}\rho_{kk}\rho_{pp}\rho_{rr}\rho_{ss}}\\
&+2\frac{\rho_{kli}\rho_{pqi}\rho_{kpr}\rho_{lqr}}{\rho_{ii}\rho_{kk}\rho_{ll}\rho_{pp}\rho_{qq}\rho_{rr}}
+4\frac{\rho_{kli}\rho_{rsi}\rho_{kpr}\rho_{lps}}{\rho_{ii}\rho_{kk}\rho_{ll}\rho_{pp}\rho_{rr}\rho_{ss}}
+2\frac{\rho_{kpr}\lll(\log m\rrr)_{kpr}}{\rho_{kk}\rho_{pp}\rho_{rr}}.
\end{aligned}
\ee
Let
\[A=\frac{\rho_{kpr}\rho_{lqr}\rho_{kli}\rho_{pqi}}{\rho_{ii}\rho_{kk}\rho_{pp}\rho_{qq}\rho_{rr}\rho_{ll}},\;\;
B=\frac{\rho_{ali}\rho_{kai}\rho_{kpr}\rho_{lpr}}{\rho_{aa}\rho_{ii}\rho_{kk}\rho_{ll}\rho_{pp}\rho_{rr}}.\]
We have
\be\label{ptd14}
\begin{aligned}
\rho^{ij}\sigma_{ij}&=\lll(2\frac{\rho_{kpr}\rho_{abr}\rho_{abkp}}{\rho_{aa}\rho_{bb}\rho_{kk}\rho_{pp}\rho_{rr}}\right.
\left.-8\frac{\rho_{kli}\rho_{lpr}\rho_{kpri}}{\rho_{ii}\rho_{kk}\rho_{ll}\rho_{pp}\rho_{rr}}\rrr)
+2\frac{\rho^2_{kpri}}{\rho_{ii}\rho_{kk}\rho_{pp}\rho_{rr}}\\
&+3B+2A+2\frac{\rho_{kpr}(\log m)_{kpr}}{\rho_{kk}\rho_{pp}\rho_{rr}}
-3\frac{\rho_{kpr}\rho_{lpr}}{\rho_{ll}\rho_{kk}\rho_{pp}\rho_{rr}}(\log m)_{kl}.
\end{aligned}
\ee
Thus,
\be\label{ptd15}
\begin{aligned}
\rho^{ij}\sigma_{ij}&=\frac{2}{\rho_{ii}\rho_{kk}\rho_{pp}\rho_{rr}}\lll(\rho^2_{kpri}\right.
\left.-3\frac{\rho_{kli}\rho_{lpr}\rho_{kpri}}{\rho_{ll}}\rrr)+3B+2A\\
&+2\frac{\rho_{kpr}(\log m)_{kpr}}{\rho_{kk}\rho_{pp}\rho_{rr}}
-3\frac{\rho_{kpr}\rho_{lpr}(\log m)_{kl}}{\rho_{ll}\rho_{kk}\rho_{pp}\rho_{rr}}\\
&=\frac{2}{\rho_{ii}\rho_{kk}\rho_{pp}\rho_{rr}}
\lll[\rho_{kpri}-\sum_l\frac{\lll(\rho_{kli}\rho_{lpr}+\rho_{pli}\rho_{klr}+\rho_{rli}\rho_{lpk}\rrr)}{2\rho_{ll}}\rrr]^2\\
&-\frac{1}{2\rho_{ii}\rho_{kk}\rho_{pp}\rho_{rr}}
\lll[\sum_l\frac{\lll(\rho_{kli}\rho_{lpr}+\rho_{pli}\rho_{klr}+\rho_{rli}\rho_{lpk}\rrr)}{\rho_{ll}}\rrr]^2\\
&+3B+2A+2\frac{\rho_{kpr}(\log m)_{kpr}}{\rho_{kk}\rho_{pp}\rho_{rr}}
-3\frac{\rho_{kpr}\rho_{lpr}(\log m)_{kl}}{\rho_{ll}\rho_{kk}\rho_{pp}\rho_{rr}}\\
&\geq 2B+(B-A)+2\frac{\rho_{kpr}(\log m)_{kpr}}{\rho_{kk}\rho_{pp}\rho_{rr}}
-3\frac{\rho_{kpr}\rho_{lpr}(\log m)_{kl}}{\rho_{ll}\rho_{kk}\rho_{pp}\rho_{rr}}.\\
\end{aligned}
\ee
Now, let $v_{kpr}=\rho_{kpr}/\lll(\rho_{kk}\rho_{pp}\rho_{rr}\rrr)^{1/2}.$ Thus, $B=v_{kpr}v_{lpr}v_{kai}v_{lai}$
and $A=v_{kpr}v_{lqr}v_{kli}v_{pqi}.$ Observe that
\be\label{ptd16}
\begin{aligned}
\frac{1}{2}\sigma^2&=\frac{1}{2}\lll(\sum_{kpr}v_{kpr}^2\rrr)^2\leq\sum_k\lll(\sum_{pr}v^2_{kpr}\rrr)^2\\
&\leq\sum_{kl}\lll(\sum_{pr}v_{kpr}v_{lpr}\rrr)^2=v_{kpr}v_{lpr}v_{kai}v_{lai}=B.
\end{aligned}
\ee
Next consider $B-A,$
\be\label{ptd17}
\begin{aligned}
0&\leq\lll(\sum_rv_{jri}v_{rkl}+v_{jrk}v_{ril}-2v_{jrl}v_{rik}\rrr)^2\\
&=\lll(v_{jri}v_{rkl}+v_{jrk}v_{ril}-2v_{jrl}v_{rik}\rrr)\lll(v_{jqi}v_{qkl}+v_{jqk}v_{qil}-2v_{jql}v_{qik}\rrr)\\
&=B-A.
\end{aligned}
\ee
Therefore,
\[\rho^{ij}\sigma_{ij}\geq\sigma^2+2\frac{\rho_{kpr}(\log m)_{kpr}}{\rho_{kk}\rho_{pp}\rho_{rr}}
-3\frac{\rho_{kpr}\rho_{lpr}(\log m)_{kl}}{\rho_{kk}\rho_{pp}\rho_{rr}\rho_{ll}}.\]
Recall that $m=K\lll(1+|\nabla u|^2\rrr)^2.$ We have
\be\label{ptd18}
\begin{aligned}
\rho^{ij}\sigma_{ij}&\geq\sigma^2+2\frac{\rho_{kpr}}{\rho_{kk}\rho_{pp}\rho_{rr}}
\lll[\log K+2\log \lll(1+|\nabla u|^2\rrr)\rrr]_{kpr}\\
&-3\frac{\rho_{kpr}\rho_{lpr}}{\rho_{kk}\rho_{pp}\rho_{rr}\rho_{ll}}
\lll[\log K+2\log\lll(1+|\nabla u|^2\rrr)\rrr]_{kl}\\
&=\sigma^2+2\frac{\rho_{kpr}}{\rho_{kk}\rho_{pp}\rho_{rr}}\lll(\log K\rrr)_{kpr}
+4\frac{\rho_{kpr}}{\rho_{kk}\rho_{pp}\rho_{rr}}\lll[\log\lll(1+|\nabla u|^2\rrr)\rrr]_{kpr}\\
&-3\frac{\rho_{kpr}\rho_{lpr}}{\rho_{kk}\rho_{pp}\rho_{rr}\rho_{ll}}\lll(\log K\rrr)_{kl}
-6\frac{\rho_{kpr}\rho_{lpr}}{\rho_{kk}\rho_{pp}\rho_{rr}\rho_{ll}}\lll[\log\lll(1+|\nabla u|^2\rrr)\rrr]_{kl}\\
&=\sigma^2+2\frac{\rho_{kpr}}{\rho_{kk}\rho_{pp}\rho_{rr}}
\lll(\frac{K_{kpr}}{K}-\frac{K_{kp}K_{r}+K_{kr}K_p+K_{pr}K_k}{K^2}+2\frac{K_kK_pK_r}{K^3}\rrr)\\
&-3\frac{\rho_{kpr}\rho_{lpr}}{\rho_{kk}\rho_{pp}\rho_{rr}\rho_{ll}}\lll(\frac{K_{kl}}{K}-\frac{K_kK_l}{K^2}\rrr)
+8\frac{\rho_{kpr}u_iu_{ikpr}}{\rho_{kk}\rho_{pp}\rho_{rr}\lll(1+|\nabla u|^2\rrr)}+S_1,
\end{aligned}
\ee
where
\be\label{ptd19}
\begin{aligned}
S_1&=4\frac{\rho_{kpr}}{\rho_{kk}\rho_{pp}\rho_{rr}}
\lll\{2\frac{(u_{ipr}u_{ik}+u_{ip}u_{ikr}+u_{ir}u_{ikp})}{1+|\nabla u|^2}\right.\\
&-4\frac{u_qu_{qr}(u_{ip}u_{ik}+u_iu_{ikp})}{\lll(1+|\nabla u|^2\rrr)^2}
-4\frac{u_lu_{lp}(u_{ir}u_{ik}+u_iu_{ikr})}{\lll(1+|\nabla u|^2\rrr)^2}\\
&-4\frac{u_iu_{ik}(u_{lr}u_{lp}+u_lu_{lpr})}{\lll(1+|\nabla u|^2\rrr)^2}
\left.+8\frac{u_{ik}u_{lp}u_{qr}u_iu_lu_q}{\lll(1+|\nabla u|^2\rrr)^3}\rrr\}\\
&-6\frac{\rho_{kpr}\rho_{lpr}}{\rho_{kk}\rho_{pp}\rho_{rr}\rho_{ll}}
\lll\{2\frac{u_{il}u_{ik}+u_iu_{ikl}}{1+|\nabla u|^2}\right.
\left.-4\frac{u_iu_{ik}u_qu_{ql}}{\lll(1+|\nabla u|^2\rrr)^2}\rrr\}.
\end{aligned}
\ee
Note that at the point $P,$ we have
\be\label{ptd20}
\sigma_i=2\frac{\rho_{ikpr}\rho_{kpr}}{\rho_{kk}\rho_{pp}\rho_{rr}}
-3\frac{\rho_{kli}\rho_{kpr}\rho_{lpr}}{\rho_{kk}\rho_{pp}\rho_{rr}\rho_{ll}},
\ee
\be\label{ptd21}
\rho_{pr}=-(uu_{pr}+u_pu_r+\delta_{pr}),
\ee
\be\label{ptd22}
\rho_{kpr}=-(u_ku_{pr}+uu_{kpr}+u_{kp}u_r+u_pu_{kr}),
\ee
\be\label{ptd23}
\begin{aligned}
\rho_{ikpr}&=-(u_{ik}u_{pr}+u_ku_{ipr}+u_iu_{kpr}+uu_{ikpr}\\
&+u_{ikp}u_r+u_{kp}u_{ir}+u_{ip}u_{kr}+u_pu_{ikr}),\\
\end{aligned}
\ee
and
\be\label{ptd24}
4\frac{u_i\sigma_i}{\lll(1+|\nabla u|^2\rrr)u}=
-8\frac{u_iu_{ikpr}}{\lll(1+|\nabla u|^2\rrr)}\frac{\rho_{kpr}}{\rho_{kk}\rho_{pp}\rho_{rr}}+S_2,
\ee
where
\be\label{ptd25}
\begin{aligned}
S_2&=4\frac{u_i\sigma_i}{\lll(1+|\nabla u|^2\rrr)u}+8\frac{u_iu_{ikpr}\rho_{kpr}}{\rho_{kk}\rho_{pp}\rho_{rr}\lll(1+|\nabla u|^2\rrr)}\\
&=4\frac{u_i}{\lll(1+|\nabla u|^2\rrr)u}\lll\{-2\frac{\rho_{kpr}}{\rho_{kk}\rho_{pp}\rho_{rr}}\right.
\lll[u_{ik}u_{pr}+u_ku_{ipr}+u_iu_{kpr}+u_{r}u_{ikp}\right.\\
&\left.+u_{kp}u_{ir}+u_{ip}u_{kr}+u_pu_{ikr}\rrr]
\left.-3\frac{\rho_{kli}\rho_{kpr}\rho_{lpr}}{\rho_{kk}\rho_{pp}\rho_{rr}\rho_{ll}}\right\}.
\end{aligned}
\ee
Therefore,
\be\label{ptd26}
\begin{aligned}
L\sigma&=\rho^{ij}\sigma_{ij}+4\frac{u_i\sigma_i}{\lll(1+|\nabla u|^2\rrr)u}\\
&\geq\sigma^2+2\frac{\rho_{kpr}}{\rho_{kk}\rho_{pp}\rho_{rr}}\lll(\frac{K_{kpr}}{K}\right.
\left.-\frac{K_{kp}K_r+K_{kr}K_p+K_kK_{pr}}{K^2}+2\frac{K_kK_pK_r}{K^3}\rrr)\\
&-3\frac{\rho_{kpr}\rho_{lpr}}{\rho_{kk}\rho_{pp}\rho_{rr}\rho_{ll}}
\lll(\frac{K_{kl}}{K}-\frac{K_kK_l}{K^2}\rrr)+S_1+S_2.\\
\end{aligned}
\ee
Recall that a point $P$ near $(0,0, u(0,0))$ has been fixed, a rotation of coordinates
has been performed about $P$ so that $\{\rho_{ij}\}$ is diagonal at $P.$ Consequently,
$\{h_{ij}\}$ is diagonal at $P$ as well. Thus,
\be\label{ptd27}
\rho_{11}\rho_{22}=w^4K=K\lll(1+|\nabla u|^2\rrr)^2,
\ee
and
\be\label{ptd28}
\rho_{11}+\rho_{22}=u^2w(h_{11}+h_{22})=2w\tilde{H}.
\ee
We denote $\tilde{H}=\frac{u^2}{2}(h_{11}+h_{22}).$

Since,
\be\label{ptd29}
\begin{aligned}
2H&=g^{11}h_{11}+g^{22}h_{22}\\
&=u^2\lll(1-\frac{u_1^2}{w^2}\rrr)h_{11}+u^2\lll(1-\frac{u_2^2}{w^2}\rrr)h_{22}\\
&=u^2(h_{11}+h_{22})-\frac{u_1^2}{w^2}u^2h_{11}-\frac{u^2_2}{w^2}u^2h_{22}\\
&\leq2\tilde{H},
\end{aligned}
\ee
we see that
\be\label{ptd30}
2H\geq\frac{2}{w^2}\tilde{H}.
\ee
We have
\be\label{ptd31}
0<\rho_{11},\;\rho_{22}\leq2w\tilde{H}\leq2w^3H.
\ee
Thus,
\be\label{ptd32}
0<\frac{1}{\rho_{ii}}=\frac{\rho_{3-i3-i}}{\rho_{ii}\rho_{3-i3-i}}
\leq\frac{2w^3H}{w^4K}=\frac{2H}{KW}.
\ee
Also, by assumption, $|\nabla u|(0,0)=0,$ so near $(0,0),$
$|\nabla u|<1.$ In the above coordinates at $P,$
\[|\nabla u|^2(P)=u_1^2+u_2^2.\]
So, we have
$|u_1|<1,$ $|u_2|<1.$ We may also assume $u\geq 1$ near $(0,0).$
Moreover, at $P,$
\[\sigma=\sum_{kpr}\frac{\rho^2_{kpr}}{\rho_{kk}\rho_{pp}\rho_{rr}}.\]
Thus for each fixed $k,p,r,$ we have
\[\frac{\rho^2_{kpr}}{\rho_{kk}\rho_{pp}\rho_{rr}}\leq\sigma,\]
and
\[\frac{|\rho_{kpr}|}{\sqrt{\rho_{kk}\rho_{pp}\rho_{rr}}}\leq\sigma^{1/2}.\]
Now, at $P,$
\[|\nabla K|^2=K_1^2+K_2^2,\]
\[|\nabla^2K|^2=K_{11}^2+2K_{12}^2+K_{22}^2,\]
\[|\nabla^3K|^2=\sum_{kpr}K^2_{kpr}.\]
Therefore, from \eqref{ptd26},
\be\label{ptd33}
\begin{aligned}
&\lll|\sum_{kpr}\frac{2\rho_{kpr}}{\rho_{kk}\rho_{pp}\rho_{rr}}\right.
\lll(\frac{K_{kpr}}{K}-\frac{(K_{kp}K_r+K_{kr}K_p+K_{pr}K_k)}{K^2}\right.
\left.\left.+\frac{2K_kK_pK_r}{K^3}\rrr)\rrr|\\
&\leq2\sum_{kpr}\frac{|\rho_{kpr}|}{\sqrt{\rho_{kk}\rho_{pp}\rho_{rr}}}\frac{1}{\sqrt{\rho_{kk}\rho_{pp}\rho_{rr}}}
\lll(\frac{|\nabla^3K|}{K}+\frac{3|\nabla^2K||\nabla K|}{K^2}+\frac{2|\nabla K|^3}{K^3}\rrr)\\
&\leq2\sigma^{1/2}\lll(\frac{|\nabla^3K|}{K}+\frac{3|\nabla^2K||\nabla K|}{K^2}\right.
\left.+\frac{2|\nabla K|^3}{K^3}\rrr)\sum_{kpr}\frac{1}{\sqrt{\rho_{kk}\rho_{pp}\rho_{rr}}}\\
&\leq2\sigma^{1/2}\lll(\frac{|\nabla^3K|}{K}+\frac{3|\nabla^2K||\nabla K|}{K^2}\right.
\left.+\frac{2|\nabla K|^3}{K^3}\rrr)\sum_{kpr}\frac{2^{3/2}H^{3/2}}{K^{3/2}w^{3/2}}\\
&=\frac{2^{11/2}H^{3/2}}{w^{3/2}}
\lll(K^{1/2}|\nabla^3K|+\frac{3|\nabla^2K||\nabla K|}{K^{1/2}}+\frac{2|\nabla K|^3}{K^{3/2}}\rrr)\frac{\sigma^{1/2}}{K^3}\\
&\leq\frac{C}{K^3}\sigma^{1/2},
\end{aligned}
\ee
where $C$ is a constant depending on the maximum of $\frac{1}{w}$ near $(0,0),$
$|K|_{C^2},$ $K^{1/2}|\nabla^3K|,$ $|\nabla K|/K^{1/2},$ and the maximum of $H.$
Next, still from \eqref{ptd26},
\be\label{ptd34}
\begin{aligned}
&\lll|-\frac{3\rho_{kpr}\rho_{lpr}}{\rho_{kk}\rho_{pp}\rho_{ll}\rho_{rr}}\right.
\left.\lll(\frac{K_{kl}}{K}-\frac{K_kK_l}{K^2}\rrr)\rrr|\\
&\leq\lll(\frac{|\nabla^2K|}{K}+\frac{|\nabla K|^2}{K^2}\rrr)
\sum_{kplr}\frac{3|\rho_{kpr}||\rho_{lpr}|}{\sqrt{\rho_{kk}\rho_{pp}\rho_{rr}}\sqrt{\rho_{rr}\rho_{pp}\rho_{ll}}}
\frac{1}{\sqrt{\rho_{kk}\rho_{ll}}}\\
&\leq3\sigma\lll(\frac{|\nabla^2K|}{K}+\frac{|\nabla K|^2}{K^2}\rrr)\times 2^5\frac{H}{Kw}\\
&\leq\frac{C}{K^2}\sigma.
\end{aligned}
\ee
Finally, estimates are to be obtained for $|S_1|$ and $|S_2|.$
From \eqref{ptd21}, \eqref{ptd22}, \eqref{ptd23}, and the assumption that $u\geq 1$ near $(0,0),$
we know at $P$ that
\[|uu_{pr}|\leq|\rho_{pr}|+2,\]
\be\label{ptd35}
|u_{pr}|\leq|\rho_{pr}|+2\leq2w^3H+2,
\ee
and
\[uu_{kpr}=-\rho_{kpr}-u_ku_{pr}-u_ru_{kp}-u_pu_{kr},\]
\be\label{ptd36}
\begin{aligned}
|u_{kpr}|&\leq|uu_{kpr}|\leq|\rho_{kpr}|+3\sum_{pr}\lll(|\rho_{pr}|+2\rrr)\\
&\leq|\rho_{kpr}|+24w^3H+24\\
&\leq\sigma^{1/2}\sqrt{\rho_{kk}\rho_{pp}\rho_{rr}}+24w^3+24\\
&\leq\sigma^{1/2}\times 2^{9/2}w^{9/2}H^{3/2}+24w^3H+24.\\
\end{aligned}
\ee
Now for $|S_1|$ we have,
\be\label{ptd37}
\begin{aligned}
|S_1|&\leq\frac{8|\rho_{kpr}|}{\sqrt{\rho_{kk}\rho_{pp}\rho_{rr}}}\frac{1}{\sqrt{\rho_{kk}\rho_{pp}\rho_{rr}}}
\sum_{ikpr}\lll(|u_{ipr}u_{ik}|+|u_{ip}u_{ikr}|+|u_{ir}u_{ikp}|\rrr)\\
&+\frac{48|\rho_{kpr}|}{\sqrt{\rho_{kk}\rho_{pp}\rho_{ll}}}\frac{1}{\sqrt{\rho_{kk}\rho_{pp}\rho_{rr}}}
\sum_{ikpqr}|u_qu_{qr}|\lll(|u_{ip}u_{ik}|+|u_iu_{ikp}|\rrr)\\
&+\frac{32|\rho_{kpr}|}{\sqrt{\rho_{kk}\rho_{pp}\rho_{rr}}}\frac{1}{\sqrt{\rho_{kk}\rho_{pp}\rho_{rr}}}
\sum_{iklpqr}|u_{ik}u_{lp}u_{qr}u_{i}u_lu_q|\\
&+\frac{12|\rho_{kpr}\rho_{lpr}|}{\sqrt{\rho_{kk}\rho_{pp}\rho_{rr}}\sqrt{\rho_{rr}\rho_{pp}\rho_{ll}}}
\frac{1}{\sqrt{\rho_{kk}\rho_{ll}}}\sum_{ikl}\lll(|u_{il}u_{ik}|+|u_iu_{ikl}|\rrr)\\
&+\frac{24|\rho_{kpr}\rho_{lpr}|}{\sqrt{\rho_{kk}\rho_{pp}\rho_{rr}}\sqrt{\rho_{rr}\rho_{pp}\rho_{ll}}}
\frac{1}{\sqrt{\rho_{kk}\rho_{ll}}}\sum_{ikql}|u_iu_qu_{ik}u_{ql}|\\
&=I+II+III+IV+V.
\end{aligned}
\ee
\be\label{ptd38}
\begin{aligned}
I&\leq8\sigma^{1/2}\times2^3\times\frac{2^{3/2}H^{3/2}}{K^{3/2}w^{3/2}}
\times2^4\\
&\times3(2w^3H+2)\lll(\sum|\rho_{ipr}|+24w^3H+24\rrr)\\
&\leq\frac{C}{K^{3/2}}\sigma^{1/2}\sum_{ipr}|\rho_{ipr}|+\frac{C}{K^{3/2}}\sigma^{1/2}\\
&\leq\frac{C}{K^{3/2}}\sigma+\frac{C}{K^{3/2}}\sigma^{1/2},\\
\end{aligned}
\ee
\be\label{ptd39}
\begin{aligned}
II&\leq48\sigma^{1/2}\times2^3\times\frac{2^{3/2}H^{3/2}}{K^{3/2}w^{3/2}}\times2^5\lll[(2w^3H+2)^3\right.\\
&\left.+(2w^3H+2)\lll(\sigma^{1/2}2^{9/2}w^{9/2}H^{3/2}+24w^3H+24\rrr)\rrr]\\
&\leq\frac{C}{K^{3/2}}\sigma^{1/2}+\frac{C}{K^{3/2}}\sigma,\\
\end{aligned}
\ee
\be\label{ptd40}
\begin{aligned}
III&\leq32\sigma^{1/2}\times2^3\times\frac{2^{3/2}H^{3/2}}{K^{3/2}w^{3/2}}\times2^6(2w^3H+2)\\
&\leq\frac{C}{K^{3/2}}\sigma^{1/2},\\
\end{aligned}
\ee
\be\label{ptd41}
\begin{aligned}
IV&\leq12\sigma\times8\frac{H}{Kw}\times2^3\lll[(2w^3H+2)^2\right.\\
&\left.+2^{9/2}w^{9/2}H^{3/2}\sigma^{1/2}+24w^3H+24\rrr]\\
&\leq\frac{C}{K}\sigma+\frac{C}{K}\sigma^{3/2},\\
\end{aligned}
\ee
\be\label{ptd42}
V\leq24\sigma\times\frac{8H}{Kw}\times2^4(2w^3H+2)^2\leq\frac{C}{K}\sigma.
\ee
Combining equations \eqref{ptd37}--\eqref{ptd42}, we get
\be\label{ptd43}
|S_1|\leq\frac{C}{K^{3/2}}\sigma+\frac{C}{K^{3/2}}\sigma^{1/2}
+\frac{C}{K}\sigma+\frac{C}{K}\sigma^{3/2}.
\ee
Similarly, we have
\be\label{ptd44}
\begin{aligned}
|S_2|&\leq\frac{8|\rho_{kpr}|}{\sqrt{\rho_{kk}\rho_{pp}\rho_{rr}}}\frac{1}{\sqrt{\rho_{kk}\rho_{pp}\rho_{rr}}}
\sum_{ikpr}\lll[3|u_{ik}u_{pr}|+4|u_ku_{ipr}|\rrr]\\
&+\sum_{klipr}
\frac{12|\rho_{kli}\rho_{kpr}\rho_{lpr}|}{\sqrt{\rho_{kk}\rho_{ii}\rho_{ll}}\sqrt{\rho_{kk}\rho_{pp}\rho_{rr}}\sqrt{\rho_{rr}\rho_{pp}\rho_{ll}}}
\sqrt{\rho_{ii}}\\
&\leq8\sigma^{1/2}\times2^{3}\times\frac{2^{3/2}H^{3/2}}{K^{3/2}w^{3/2}}\times2^4\lll[3(2w^2H+2)^2\right.\\
&\left.+4\lll(\sigma^{1/2}2^{9/2}w^{9/2}H^{3/2}+24w^3H+24\rrr)\rrr]\\
&+12\sigma^{3/2}\times2^5\times2^{1/2}w^{3/2}H^{1/2}\\
&\leq\frac{C}{K^{3/2}}\sigma^{1/2}+\frac{C}{K^{3/2}}\sigma+C\sigma^{3/2}.
\end{aligned}
\ee
Therefore,
\be\label{ptd45}
\begin{aligned}
L\sigma&\geq\sigma^2-\frac{C}{K^3}\sigma^{1/2}-\frac{C}{K^2}\sigma-\frac{C}{K^{3/2}}\sigma\\
&-\frac{C}{K^{3/2}}\sigma^{1/2}-\frac{C}{K}\sigma-\frac{C}{K}\sigma^{3/2}-C\sigma^{3/2}.\\
\end{aligned}
\ee
Moreover,
\begin{align*}
\frac{\sigma^{1/2}}{K^3}&\leq\frac{\delta^4}{4}\sigma^2+\frac{3}{4\delta^{4/3}}\times\frac{4}{K^4},\\
\frac{\sigma}{K^2}&\leq\frac{\delta^2}{2}\sigma^2+\frac{1}{2\delta^2}\times\frac{1}{K^4},\\
\frac{\sigma}{K^{3/2}}&\leq\frac{\sigma^{1/2}}{K^3}+\sigma^{3/2},\\
\sigma^{3/2}&\leq\delta\sigma^2+\frac{\sigma}{\delta K^2},\\
\frac{\sigma}{K}&\leq\delta\sigma^2+\frac{1}{\delta K^2},\\
\frac{\sigma^{1/2}}{K^{3/2}}&\leq\frac{\sigma}{K}+\frac{1}{K^2},\\
\frac{\sigma^{3/2}}{K}&\leq\delta\sigma^2+\frac{\sigma}{\delta K^2}.\\
\end{align*}
Therefore, choose $\delta>0$ small, we have
\be\label{ptd46}
L\sigma\geq\frac{\sigma^2}{2}-\frac{C}{K^4}.
\ee
\end{proof}

Following the idea of \cite{Ia92}, we obtain an estimate for $K^2\sigma$ and $K|\nabla H|$ near $(0,0)$
and finally prove Theorem \ref{Intt2}.
\begin{lemma}\label{ptdl2}
Let $X: (S^2, g)\rightarrow (\mathbb{H}^3, h)$ be a $C^2$ isometric embedding, and $g\in C^5$ is a metric with positive Gauss curvature $K.$
Then
\[K^2\sigma\leq C,\;K|\nabla H|\leq C,\]
where $C$ is a constant depending only on the maximum of $\frac{1}{w}$ near $(0,0),$
the maximum of $|K|_{C^2},$ the maximum of $K^{1/2}|\nabla^3K|,$ the maximum of $|\nabla^2K|/K,$
and the maximum of mean curvature $H.$
\end{lemma}
\begin{proof}
Recall that
\be\label{ptd47}
\frac{\det \mbox{Hess}\rho}{\det g}=u^4w^2K.
\ee
Without loss of generality, it's assumed that $|\nabla u|(0,0)=0$
and that the inner normal of embedding is chosen so that the mean curvature $H$
is positive. Now let
\be\label{ptd48}
0\leq f=K^2\sigma.
\ee
If $f$ has its maximum away from $(0,0),$ then an estimate follows since
the equation \eqref{ptd47} is uniformly elliptic there. So assume $f$
achieves its maximum at $Q,$ where $Q$ is in some small ball centered at $(0,0).$
Then,
\[f_i(Q)=0,\;\{f_{ij}(Q)\}\leq0\]
holds at $Q.$ Since $\{\rho^{ij}\}>0,$
\be\label{ptd49}
\rho^{ij}f_{ij}(Q)\leq 0.
\ee
Computing, we have
\be\label{ptd50}
f_i=2KK_i\sigma+K^2\sigma_i,
\ee
\be\label{ptd51}
f_{ij}=2K_iK_j\sigma+2KK_{ij}\sigma+2KK_i\sigma_j+2KK_j\sigma_i+K^2\sigma_{ij}.
\ee
Combining \eqref{ptd50} and \eqref{ptd51} we get,
\be\label{ptd52}
f_{ij}=2KK_{ij}\sigma-6K_iK_j\sigma+K^2\sigma_{ij}.
\ee
Therefore,
\be\label{ptd53}
0\geq\rho^{ij}f_{ij}=2\rho^{ij}K_{ij}K\sigma-6\rho^{ij}K_iK_j\sigma+K^2\rho^{ij}\sigma_{ij},
\ee
and
\be\label{ptd54}
K^4\rho^{ij}\sigma_{ij}\leq6\rho^{ij}K_iK_jf-2\rho^{ij}K_{ij}Kf.
\ee
Next, \eqref{ptd1} may be rewritten as
\be\label{ptd55}
\begin{aligned}
K^4L\sigma&=K^4\rho^{ij}\sigma_{ij}+\frac{4K^4u_i\sigma_i}{(1+|\nabla u|^2)u}\\
&\geq \frac{1}{2}f^2-C.\\
\end{aligned}
\ee
Thus,
\be\label{ptd56}
\begin{aligned}
f^2&\leq2C+2K^4\rho^{ij}\sigma_{ij}+\frac{8K^4u_i\sigma_i}{(1+|\nabla u|^2)u}\\
&\leq2C+12\rho^{ij}K_iK_jf-4\rho^{ij}K_{ij}Kf-\frac{16K^3u_iK_i}{(1+|\nabla u|^2)u}.\\
\end{aligned}
\ee
Since $|\nabla u|<1$ near $(0,0),$ we have $|u_iK_i|\leq|\nabla K|.$
Moreover, at $Q$,
\be\label{ptd57}
\begin{aligned}
0&\leq\rho^{ij}K_iK_j=\rho^{11}K_1^2+\rho^{22}K_2^2\\
&=\frac{K_1^2}{\rho_{11}}+\frac{K_2^2}{\rho_{22}}\leq\frac{2H}{Kw}\lll(K_1^2+K_2^2\rrr)\\
&=\frac{2H}{w}\frac{|\nabla K|^2}{K}.\\
\end{aligned}
\ee
Similarly,
\be\label{ptd58}
\begin{aligned}
\lll|\rho^{ij}K_{ij}K\rrr|&=\frac{K_{11}K}{\rho_{11}}+\frac{K_{22}K}{\rho_{22}}\\
&\leq\frac{2H}{w}|\nabla^2K|.
\end{aligned}
\ee
So, we have
\be\label{ptd59}
f^2(Q)\leq C+Cf(Q).
\ee
This shows that $K^2\sigma$ is bounded.

Next, we are going to estimate $K|\nabla H|.$
We know that
\be\label{ptd60}
2H=\frac{1}{w}\lll(\delta_{ij}-\frac{u_iu_j}{w^2}\rrr)\rho_{ij}.
\ee
Differentiating and multiplying by $K$ we get
\be\label{ptd61}
2K\nabla H=K\nabla\lll(\frac{1}{w}\lll(\delta_{ij}-\frac{u_iu_j}{w^2}\rrr)\rrr)\rho_{ij}
+\frac{1}{w}\lll(\delta_{ij}-\frac{u_iu_j}{w^2}\rrr)K\rho_{ijk}.
\ee
Since the first term is bounded by $|X|_{C^2},$ we only need to find a bound for the second term.
\be\label{ptd62}
\begin{aligned}
\lll|\frac{1}{w}\lll(\delta_{ij}-\frac{u_iu_j}{w^2}\rrr)K\rho_{ijk}\rrr|^2
&\leq\max_{i,j}\lll|\frac{1}{w}\lll(\delta_{ij}-\frac{u_iu_j}{w^2}\rrr)\rrr|^2
\lll(\sum_{i,j,k}|K\rho_{ijk}|\rrr)^2\\
&\leq\max_{i,j}\lll|\frac{1}{w}\lll(\delta_{ij}-\frac{u_iu_j}{w^2}\rrr)\rrr|^2
8\lll(\sum_{i,j,k}K^2\rho_{ijk}^2\rrr)\\
&=\max_{i,j}\lll|\frac{1}{w}\lll(\delta_{ij}-\frac{u_iu_j}{w^2}\rrr)\rrr|^2
8\lll(\sum_{i,j,k}K^2\frac{\rho_{ijk}^2}{\rho_{ii}\rho_{jj}\rho_{kk}}\rho_{ii}\rho_{jj}\rho_{kk}\rrr)\\
&\leq\max_{i,j}\lll|\frac{1}{w}\lll(\delta_{ij}-\frac{u_iu_j}{w^2}\rrr)\rrr|^2
8K^2\sigma\times8w^9H^3.
\end{aligned}
\ee
We can see that every term on the right hand side is bounded. Therefore, we have $K|\nabla H|$ is bounded.
\end{proof}

\subsection{Proof of Theorem 1.2}
\label{spot}
Now, we are ready to prove Theorem \ref{Intt2}. The proof is the same as in the Euclidean case (see \cite{Ia92}).
For completeness, we include it here.

From Theorem \ref{Intt1}, we know that there exists $X_0\in C^{1,1}$ realizing $g^0.$ Now choose a sequence of
$C^{\infty}$ isometric embedding $\{X_\epsilon\}$ and corresponding metrics $\{g^\epsilon\}$ such that
\[\|g^\epsilon-g^0\|_{C^5(g^0)}\rightarrow 0,\]
\[\|X_\epsilon-X_0\|_{C^{1,1}(g^0)}\rightarrow 0,\]
\[\mbox{the extrinsic Gauss curvature $K^\epsilon>0$,}\]
and
\[ H^\epsilon\geq \frac{c}{2},\]
where by assumption $0<c=\lim\inf_{Q\rightarrow P_i}H(Q).$ Then, from Lemma \ref{ptdl2},
\be\label{spot1}
K^\epsilon|\nabla_{g^\epsilon}H^\epsilon|_{g^\epsilon}\leq C
\ee
on $B_{r_i}(P_i)$ where $r_i$ and $C$ are independent of $\epsilon.$
Also, on $S^2\backslash B_{r_i}(P_i),$ the $\{K^\epsilon\}$ are bounded from below.
Therefore, by standard elliptic theory,
\be\label{spot2}
K^\epsilon|\nabla_{g^\epsilon}H^\epsilon|_{g^\epsilon}\leq C
\ee
on $S^2\backslash B_{r_i}(P_i)$ and $C$ is independent of $\epsilon.$
Hence,
\be\label{spot3}
K^\epsilon|\nabla_{g^\epsilon}H^\epsilon|_{g^\epsilon}\leq C
\ee
on $S^2.$

Now, since $H^\epsilon\geq c/2>0,$ there exists $r_2>0$ independent of $\epsilon$
such that in $B_{r_i}(P_i)$
\be\label{spot4}
(H^\epsilon)^2-K^\epsilon\geq c_0>0.
\ee
Hence,
\[\kappa_{1,\epsilon}=H^\epsilon-\sqrt{(H^\epsilon)^2-K^\epsilon}\]
is differentiable,
and
\[\nabla_{g^\epsilon}\kappa_{1,\epsilon}
=\frac{-K^\epsilon\nabla_{g^\epsilon}H^\epsilon}{\sqrt{(H^\epsilon)^2-K^\epsilon}\lll(H^\epsilon+\sqrt{(H^\epsilon)^2-K^\epsilon}\rrr)}
+\frac{\nabla_{g^\epsilon}K^\epsilon}{2\sqrt{(H^\epsilon)^2-K^\epsilon}}.\]
From \eqref{spot3} and \eqref{spot4} it's easy to see that
\[|\nabla_{g^\epsilon}\kappa_{1,\epsilon}|\leq C\]
where $C$ is independent of $\epsilon.$

 \bigskip

\section{A priori bounds in n-dimensions}\label{apb}
\subsection{Formulas on hypersurfaces and some basic identities}
\label{foh}
In this section we recall some basic identities on a hypersurface that were derived in \cite{GS11}
by comparing the induced hyperbolic and Euclidean metrics. In the following, we identify
$\mathbb{H}^{n+1}$ with the upper half-space model.

Let $\Sigma$ be a hypersurface in $\mathbb{H}^{n+1}$. We shall use $g$ and $\nabla$
to denote the induced hyperbolic metric and Levi-Civita connection
on $\Sigma$, respectively. As $\Sigma$ is also a submanifold of $\mathbb{R}^{n+1}$,
we shall usually distinguish a geometric quantity with respect to the
Euclidean metric by adding a `tilde' over the corresponding hyperbolic
quantity.
For instance, $\tg$ denotes the induced  metric on $\Sigma$
from $\mathbb{R}^{n+1}$, and $\tnabla$ is its Levi-Civita connection.

Let $\vx$ be the position vector of $\Sigma$ in $\mathbb{R}^{n+1},$ and set
\[ u = \vx \cdot \bf{e} \]
where $\bf{e}$ is the unit vector in the positive
$x_{n+1}$ direction in $\mathbb{R}^{n+1}$, and `$\cdot$' denotes the Euclidean
inner product in $\mathbb{R}^{n+1}$.
We refer to $u$ as the {\em height function} of $\Sigma$.

We assume $\Sigma$ is orientable and let ${\bf n}$ be
a (global)
unit normal vector field to $\Sigma$ with respect to the hyperbolic metric.
This also determines a unit normal $\nu$ to $\Sigma$ with respect to the
Euclidean metric by the relation
\[ \nu = \frac{{\bf n}}{u}. \]
We denote $ \nu^{n+1} = \bf{e} \cdot \nu$.

Let $(z_1, \ldots, z_n)$ be local coordinates and
\[ \tau_i = \frac{\partial}{\partial z_i}, \;\; i = 1, \ldots, n. \]
The hyperbolic and Euclidean metrics of $\Sigma$ are given by
\[ g_{ij} = \langle \tau_i, \tau_j \rangle, \;\;
   \tg_{ij} = \tau_i \cdot \tau_j = u^2 g_{ij}. \]
The second fundamental forms are
\begin{equation}
\label{eq2.10}
 \begin{aligned}
   h_{ij} \,& = \langle D_{\tau_i} \tau_j, {\bf n} \rangle
              = - \langle D_{\tau_i} {\bf n}, \tau_j \rangle, \\
\thh_{ij} \,& = \nu \cdot \tD_{\tau_i} \tau_j
              = - \tau_j \cdot \tD_{\tau_i} \nu,
  \end{aligned}
\end{equation}
where $D$ and $\tD$ denote the Levi-Civita connection of $\mathbb{H}^{n+1}$
and $\mathbb{R}^{n+1}$, respectively.
 The following relations are well known (see \eqref{dmhf5}):
\begin{equation}
\label{eq2.20}
 h_{ij} = \frac{1}{u} \thh_{ij} + \frac{\nu^{n+1}}{u^2} \tg_{ij}
 \end{equation}
and 
\begin{equation}
\label{eq2.30}
\kappa_i = u \tilde{\kappa}_i + \nu^{n+1}, \;\;\; i = 1, \cdots, n
\end{equation}
where $\kappa_1, \cdots, \kappa_n$ and
$\tilde{\kappa}_1, \cdots, \tilde{\kappa}_n$ are the hyperbolic and
Euclidean principal curvatures, respectively.
The Christoffel symbols are related by the formula
\begin{equation}
\label{eq2.40}
\Gamma_{ij}^k = \tilde{\Gamma}_{ij}^k - \frac{1}{u}
   (u_i \delta_{kj} + u_j \delta_{ik} - \tg^{kl} u_l \tg_{ij}).
\end{equation}
It follows that for $v \in C^2 (\Sigma)$
\begin{equation}
\label{eq2.50}
\nabla_{ij} v = v_{ij} - \Gamma_{ij}^k v_k
  = \tilde{\nabla}_{ij} v + \frac{1}{u}
    (u_i v_j + u_j v_i - \tg^{kl} u_l v_k \tg_{ij})
\end{equation}
where
\[ v_i = \frac{\partial v}{\partial z_i}, \;
   v_{ij} = \frac{\partial^2 v}{\partial z_i z_j}, \; \mbox{etc.} \]
In particular,
\begin{equation}
\label{eq2.60}
\begin{aligned}
\nabla_{ij} u
 \,& = \tilde{\nabla}_{ij} u + \frac{2 u_i u_j}{u}
       - \frac{1}{u} \tg^{kl} u_k u_l \tg_{ij}
  \end{aligned}
\end{equation}
and
\begin{equation}
\label{eq2.70}
\nabla_{ij} \frac{1}{u}
  = - \frac{1}{u^2} \tilde{\nabla}_{ij} u
  + \frac{1}{u^3} \tg^{kl} u_k u_l \tg_{ij}.
\end{equation}
Moreover,
\begin{equation}
\label{eq2.80}
\begin{aligned}
\nabla_{ij} \frac{v}{u}
 \,&  = v \nabla_{ij} \frac{1}{u}
        + \frac{1}{u} \tilde{\nabla}_{ij} v
        - \frac{1}{u^2} \tg^{kl} u_k v_l \tg_{ij}.
 \end{aligned}
\end{equation}

In $\mathbb{R}^{n+1}$,
\begin{equation}
\label{eq2.90}
\begin{aligned}
             \tg^{kl} u_k u_l
      \,& = |\tilde{\nabla} u|^2 = 1 - (\nu^{n+1})^2 \\
            \tilde{\nabla}_{ij} u
      \,& = \thh_{ij} \nu^{n+1}.
 \end{aligned}
\end{equation}
Therefore, by \eqref{eq2.30} and \eqref{eq2.70},
\begin{equation}
\label{eq2.100}
\begin{aligned}
\nabla_{ij} \frac{1}{u}
 \,&  = - \frac{\nu^{n+1}}{u^2} \thh_{ij}
  + \frac{1}{u^3} (1 - (\nu^{n+1})^2) \tg_{ij} \\
 \,&  = \frac{1}{u} (g_{ij} - \nu^{n+1} h_{ij}).
 \end{aligned}
\end{equation}
We note that \eqref{eq2.80} and \eqref{eq2.100} still hold for
general local frames $\tau_1, \ldots, \tau_n$.

\subsection{Proof of Theorem 1.3}\label{apbp}
For simplicity, in this subsection we are going to do calculations in the upper half-space model.

Let $g$ be a $C^4$ metric of nonnegative sectional curvature on $S^n,$
and $X:(S^n, g)\rightarrow \mathbb{H}^{n+1}$ be a $C^4$ isometric embedding into the hyperbolic
space $\mathbb{H}^{n+1}$ (we assume that the critical points are finite and isolated).
Then as in \cite{Po64}, for any point $P\in X(S^n),$ we can always construct a convex cap
$\omega,$ such that $P\in \omega$ and $\partial\omega$ is a $n-1$ dimensional submanifold with positive curvatures.
Using the techniques introduced in Section \ref{dmh}, we can always assume that $\omega$
can be represented as a graph over $\mathbb{R}^n\times\{0\}$ in the half-space model, and on $\partial\omega$
the height function $u\equiv c>0.$

When $\omega$ is strictly convex, we can apply results of \cite{GSS09} directly and obtain a priori bounds
for the principal curvatures which only depend on the metric $g$ on $X^{-1}(\omega)\subset S^n.$

When there exists a critical point, say $P_0\in\omega.$ By lifting $\omega=\{(x, u(x))|\; x\in \mathbb{R}^n\}$
to $\omega^\epsilon=\{(x, u(x)+\epsilon)|\; x\in \mathbb{R}^n\},$ $\epsilon>0$ small, we obtain a
sequence of $g^\epsilon\in C^4(X^{-1}(\omega))$ with positive sectional curvatures which are isometrically
embeddable  in $\mathbb{H}^{n+1},$ and $g^\epsilon$ converges to $g$ on $X^{-1}(\omega).$

Now, we are going to generalize the result we proved in Section \ref{mce} to $n$-dimensions ($n\geq 3$), and in the following,
for convenience, we drop the dependence on $\epsilon$ in our notation.

\begin{theorem}\label{apbt}
Let $g$ be a $C^4$ metric with sectional curvature $\geq -1$ and let $X: (S^n, g)\rightarrow \mathbb{H}^{n+1}$ be a $C^4$ isometric embedding.
Suppose the sectional curvature $S$ of $g$ satisfies\\
(1) $S(P_i,\chi)=-1$ for some $\chi\in\bigwedge^2 T_{P_i}M$; $1\leq i\leq n,$\\
(2) $S(Q)>-1,$ for any $Q\neq P_i,$ where $\{P_i\}\in S^n$ are finite isolated points.\\
Let $H$ be the trace of the second fundamental form of $X,$ and
let $R$ be the extrinsic scalar curvature of $g.$ Then the following inequality holds:
\be\label{apb1}
H^2\leq C_1|\Delta R|+C_2(R^2+R),
\ee
where $C_1,$ $C_2$ only depends on the metric $g$ and dimension $n.$
\end{theorem}
\begin{proof}
Similar to Section \ref{mce}, we consider a function $f=e^{\frac{\alpha}{u}}H$ in $B_r(X(P_0))$ with $\alpha>0$ to be chosen later. By the procedure introduced in Section
\ref{dmh}, we can assume $u(X(P_0))>1$ and $\nabla u(X(P_0))=0$ in the upper half-space model.
Choosing $r>0$ small enough, we have $\lll(\nu^{n+1}\rrr)^2>1-\lll(\min u/4\max u\rrr)^2$ in $B_r(X(P_0)).$
Assume $f$ achieves its maximum at an interior point $Q\in B_r(X(P_0)).$ We choose a local orthonormal frame
$\tau_1,\cdots, \tau_n$ around $Q.$ For convenience we shall write
$v_{ij}=\nabla_{ij}v,\;h_{ijk}=\nabla_kh_{ij},\;h_{ijkl}=\nabla_{lk}h_{ij},$ etc.
Then, at this point, we have
\be\label{apb2}
f_i=\alpha\lll(\frac{1}{u}\rrr)_iH+H_i=0,
\ee
\be\label{apb3}
f_{ij}=e^{\frac{\alpha}{u}}\lll(\alpha H\lll(\frac{1}{u}\rrr)_{ij}\right.
\left.-\alpha^2H\lll(\frac{1}{u}\rrr)_i\lll(\frac{1}{u}\rrr)_j+H_{ij}\rrr)\leq0.
\ee
Since $R=H^2-trA^2,$ by \eqref{apb2}, we have
\be\label{abp4}
\begin{aligned}
\Delta R&=2H\Delta H+2|\nabla H|^2-2|\nabla A|^2-2h^{ij}\Delta h_{ij}\\
&\leq2H\Delta H+2\alpha^2\lll|\nabla\lll(\frac{1}{u}\rrr)\rrr|^2H^2-2h^{ij}H_{ij}
+2h^{ij}\lll(R_{jlik}h^{lk}+R^k_{lik}h^l_j\rrr).\\
\end{aligned}
\ee
Also,
\be\label{apb01}
\begin{aligned}
&2(trA^2)^2-2trA^3trA\\
&=2(H^2-R)^2-2trA^3trA\\
&\leq 2(H^2-R)^2-trA\lll\{2(trA)^3-3[(trA)^2-trA^2]trA\rrr\}\\
&\leq -H^2R+2R^2,
\end{aligned}
\ee
where we used
$3\lll(\sum x_i^2\rrr)(\sum x_i)\leq\lll(\sum x_i\rrr)^3+2\sum x_i^3,$ for $x_i>0,\; 1\leq i\leq n.$

Denote $d^2=1-\lll(\nu^{n+1}\rrr)^2.$ Combining \eqref{apb01} and \eqref{abp4}, we get
\be\label{apb5}
\begin{aligned}
\Delta R&\leq2H\Delta H-2h^{ij}H_{ij}+2\frac{\alpha^2d^2}{u^2}H^2
+2h^{ij}\lll[\lll(h_{ji}h_{lk}-h_{jk}h_{li}\rrr)\right.\\
&\left.-\lll(g_{ji}g_{lk}-g_{jk}g_{li}\rrr)\rrr]h^{lk}+2h^{ij}
\lll[\lll(h^k_ih_{lk}-h^k_kh_{li}\rrr)-(g^k_ig_{lk}-g^k_kg_{li})\rrr]h^l_j\\
&=2H\Delta H-2h^{ij}H_{ij}+2\frac{\alpha^2d^2}{u^2}H^2+2\lll[(trA^2)^2-trA^3trA\rrr]\\
&-2h^{ij}g_{ij}g_{lk}h^{lk}+2h^{ij}g_{jk}g_{li}h^{lk}-2h^{ij}g^k_ig_{lk}h^l_j+2nh^{ij}g_{li}h^l_j\\
&\leq2H\Delta H-2h^{ij}H_{ij}+2\frac{\alpha^2d^2}{u^2}H^2-RH^2+2R^2-2H^2+2n(H^2-R)\\
&\leq2(Hg^{ij}-h^{ij})H_{ij}+2\frac{\alpha^2d^2}{u^2}H^2-RH^2+(2n-2)H^2+C_2(R^2+R),\\
\end{aligned}
\ee
where we used \eqref{eq2.90}.

At $Q,$
\[\lll(Hg^{ij}-h^{ij}\rrr)f_{ij}\leq 0.\]
We have
\be\label{abp6}
\begin{aligned}
\lll(Hg^{ij}-h^{ij}\rrr)H_{ij}&\leq\lll(Hg^{ij}-h^{ij}\rrr)
\lll(-\alpha H\lll(\frac{1}{u}\rrr)_{ij}+\alpha^2H\lll(\frac{1}{u}\rrr)_i\lll(\frac{1}{u}\rrr)_j\rrr)\\
&\leq-\alpha H\lll(Hg^{ij}-h^{ij}\rrr)\lll(\frac{1}{u}g_{ij}-\frac{\nu^{n+1}}{u}h_{ij}\rrr)
+\alpha^2H^2g^{ij}\lll(\frac{1}{u}\rrr)_i\lll(\frac{1}{u}\rrr)_j\\
&\leq\frac{-\alpha H^2(n-1)}{u}+\frac{\alpha HR}{u}\nu^{n+1}+\frac{\alpha^2d^2H^2}{u^2}.\\
\end{aligned}
\ee
Therefore,
\be\label{abp7}
\begin{aligned}
\Delta R&\leq\frac{-2\alpha(n-1)}{u}H^2+\frac{2\alpha HR}{u}\sqrt{1-d^2}\\
&+\frac{4\alpha^2d^2H^2}{u^2}-RH^2+(2n-2)H^2+C_2(R^2+R).\\
\end{aligned}
\ee
By
\[RH\lll(2\alpha u\sqrt{1-d^2}-H\rrr)\leq R\alpha^2u^2(1-d^2)\]
we have
\be\label{apb8}
\begin{aligned}
u^2\Delta R&\leq-2\alpha u(n-1)H^2+4\alpha^2d^2H^2+(2n-2)u^2H^2\\
&+\alpha^2u^2(1-d^2)R+C_2(R^2+R).\\
\end{aligned}
\ee
Choosing $\alpha=\frac{n-1}{\min_{X(S^n)}u}\lll(\max_{X(S^n)}u\rrr)^2,$ we have
\be\label{apb9}
u^2\Delta R\leq\frac{-1}{2}(n-1)^2H^2+C_2(R^2+R).
\ee
Thus
\be\label{abp10}
H^2\leq C_1|\Delta R|+C_2(R^2+R).
\ee
\end{proof}

\bigskip
\footnotesize
\noindent{\bf{Acknowledgments.}}
The authors would like to thank our advisors, Professor Joel Spruck and Professor William Minicozzi, for their support and guidance.

\end{document}